\documentclass[11pt,oneside]{amsart}

\usepackage[margin=1in]{geometry}

\usepackage{amssymb, amsmath}
\usepackage{float}
\makeatletter
\@namedef{subjclassname@2020}{\textup{2020} Mathematics Subject Classification}
\makeatother

\usepackage{comment}
\usepackage{booktabs}
\usepackage{graphicx}
\usepackage{color}
\usepackage{pinlabel}
\usepackage{mathtools,xparse}
\usepackage{tikz}
\usetikzlibrary{arrows,positioning,shapes,fit,calc}

\usepackage{enumerate}

\usepackage[colorlinks = true,
            linkcolor = red,
            urlcolor  = red,
            citecolor = red,
            anchorcolor = red]{hyperref}
\usepackage{amsrefs}

\usepackage[pagewise]{lineno}

\usepackage{pinlabel}

\theoremstyle{plain}
\newtheorem{theorem}{Theorem}[section]

\newtheorem*{theorem*}{Theorem}

\newtheorem*{maintheorem-intro}{Theorem}
\newtheorem*{maintheorem-intro-2}{Theorem~\ref{Bridge number and genus}}
\newtheorem*{theorem-cablingconj}{Theorem~\ref{apps1} (1)}
\newtheorem*{theorem-toroidal}{Specialization of Theorem~\ref{apps1} (2)}
\newtheorem*{theorem-lens}{Theorem~\ref{Bounding distance - special}(1)}
\newtheorem*{theorem-SFS}{Theorem~\ref{Bounding distance - special}(2)}
\newtheorem*{theorem-cosmetic}{Theorem}
\newtheorem*{theorem-bridge}{Specialization of Corollary~\ref{Cor: exceptional bridge}}
\newtheorem*{theorem-Heeggenus}{Corollary~\ref{Cor: exceptional bridge} (2)}

\newtheorem{proposition}[theorem]{Proposition}
\newtheorem{corollary}[theorem]{Corollary}
\newtheorem{lemma}[theorem]{Lemma}

\theoremstyle{definition}
\newtheorem{remark}[theorem]{Remark}
\newtheorem{definition}[theorem]{Definition}

\newtheorem{example}[theorem]{Example}

\theoremstyle{definition}

\newcommand{\bi}{\begin{itemize}}
\newcommand{\ei}{\end{itemize}}
\newcommand{\be}{\begin{enumerate}}
\newcommand{\ee}{\end{enumerate}}

 \usepackage[abs]{overpic}
\begin{document}
   \title[The Shadow Quandle Cocycle Invariant of Knotoids]{The Shadow Quandle Cocycle Invariant of Knotoids}
   \author{Nicholas Cazet}
   \begin{abstract}

   This paper studies the chirality of knotoids using shadow quandle colorings and the shadow quandle cocycle invariant. The shadow coloring number and the shadow quandle cocycle invariant is shown to distinguish infinitely many knotoids from their mirrors. Specifically, the knot-type knotoid $3_1$ is shown to be chiral.  The weight of a quandle 3-cocycle is used to calculate the crossing numbers of infinitely many multi-linkoids.

 \end{abstract}

\maketitle

\section{Introduction}

Since their introduction in 2012 by Turaev \cite{turaev2012knotoids}, knotoids and their invariants have generated much interest, for example see \cites{kodokostas2019rail,gugumcu2017new,manouras2021finite,goundaroulis2019systematic}. G\"ug\"umc\"u, Kauffman, and Lambropoulou give an excellent account of  knotoid research from its inception until 2019 in \cite{gugumcu2017knotoids}. Included in their exposition is the interest in knotoids among mathematical biologists for modeling open protein chains, for example see \cites{barbensi2021f,dorier2018knoto,goundaroulis2017studies,polym}.

Quandles were independently introduced by Joyce and Matveev in the 1980's \cites{joyce,Matveev_1984}. They showed that the fundamental knot quandle is an algebraic invariant that distinguishes knots up to mirror reflection. The texts \cites{kamada2017surface,CKS,elhamdadi2015quandles,nosaka2017quandles } give a good account of definitions and applications of quandles in (surface-) knot theory.  The first formal study of quandles in the context of knotoids was done by  G\"ug\"umc\"u and Nelson in \cite{gugumcu2019biquandle}. They showed that the biquandle counting invariant can detect mirror images of knotoids. Later, G\"ug\"umc\"u, Nelson, and Oyamguchi introduced the biquandle bracket of knotoids and showed that this invariant is stronger than the counting invariant \cite{gugumcu2021biquandle}.

 Shadow quandle coloring and the shadow quandle cocycle invariant were defined in \cite{shadowquandle}. In fact, a version of shadow quandle coloring of knotoid diagrams was defined, but the diagrams were called 1-knot arcs and the color of an endpoint arc equaled the color of its adjacent region. This was defined in the context of cross-sections of surface-links. This paper generalizes and reformulates the original definitions of shadow quandle coloring and the shadow quandle cocycle invariant of 1-knot arcs in the context of knotoids.

 The shadow cocycle invariant is used to distinguishes the knot-type knotoid $3_1$ from its mirror in Theorem \ref{thm:trefoil}.  Moreover, the shadow coloring number and shadow cocycle invariant are used to distinguish infinitely many knotoids from their mirrors in Theorems \ref{thm:colornumber} and \ref{thm:cocycle}.

Section \ref{sec2} introduces knotoids and multi-linkoids, Section \ref{sec3} defines the shadow coloring number of knotoids and uses it to distinguish infinitely many knotoids from their mirrors, Section \ref{sec4} defines the shadow quandle cocycle invariant for knotoids, and Section \ref{sec5} proves several chirality and crossing number results of knotoids and multi-linkoids using the invariant.

\section{Knotoids}
\label{sec2}

Knotoids are equivalence classes of knotoid diagrams.

\begin{definition} A {\it knotoid diagram} $\kappa$ in an orientable surface $\Sigma$ is a generic immersion of the unit interval with crossing information given at each double point. Additionally, the image of 0 and 1 are distinct and called the {\it head} and {\it leg} of $\kappa$. A {\it trivial}  knotoid diagram is an embedding of the unit interval. 

\end{definition}

The three Reidemeister moves are defined on knotoid diagrams away from the head and leg, the supporting disk of the move is disjoint from either endpoint.

\begin{definition}

Consider the equivalence relation on knotoid diagrams generated by Reidemeister moves and planar isotopy.  The equivalence classes of this relation are called {\it knotoids}. Knotoids in $\Sigma=S^2$ and $\Sigma=\mathbb{R}^2$ are referred to as {\it classical}.

\end{definition}

\begin{definition}

A classical knotoid is called a {\it knot-type knotoid} if it admits a diagram with both endpoints in the outer region. 

\end{definition}

In $S^2$, a knotoid is a knot-type knotoid if it admits a diagram with both endpoints in the same region. Each knotoid has an associated knot by connecting the endpoints of any representative diagram by any generic simple arc that is given under crossing information wherever it meets the knotoid diagram. When a classical knotoid is knot-type then its invariants are closely related to the invariants of its associated knot.

\begin{definition}[Turaev '12 \cite{turaev2012knotoids}]

Each knotoid diagram has a closed 2-disk neighborhood $B\subset\Sigma$ that meets the diagram along a radius of $B$. This neighborhood is called a {\it regular neighborhood} of the endpoints. Given knotoid diagrams $D$ and $D'$ representing the knotoids $\kappa$ and $\kappa'$, pick regular neighborhoods $B$ and $B'$ of the leg of $\kappa$ and the head of $\kappa'.$ Glue $\Sigma-\text{Int}(B)$ to $\Sigma^2-\text{Int}(B')$ along a homeomorphism $\partial B \to \partial B'$ carrying the only point of $D\cap\partial B_1$ to the only point of $D'\cap\partial B_2$. This creates the {\it product knotoid diagram} $DD'$ in $\Sigma\#\Sigma$. Moreover, this product is invariant of diagram, i.e. $\kappa\kappa':=DD'$ is a well-defined {\it knotoid product} for any diagrams $D$ and $D'$ of $\kappa$ and $\kappa'$.
\end{definition}

Turaev also described that when the knotoids $\kappa\subset S^2$ and $\kappa'\subset S^2$ have diagrams $D$ and $D'$ such that the head of $D'$ is adjacent to the outer region, a diagram of $\kappa\kappa'$ is given by concatenating $D$ and $D'$ at the leg of $D$ and head of $D'$ in a neighborhood of the leg of $D$.  This concatenation gives a knotoid diagram $DD'$ representing $\kappa\kappa'$.

\begin{definition}
A knotoid is {\it prime} if it cannot be written as a product of two non-trivial knotoids.

\end{definition}

Goundaroulis, Dorier, and Stasiak systematically classified prime classical knotoids up to five crossings \cite{goundaroulis2019systematic}. This classification is up to orientation and the symmetry-related involutions, including mirror reflection. They introduced a two-number notation of knotoids, analogous to the Rolfsen notation of  knots and links. The first number represents the crossing number of the knotoid  and the second number represents an index to distinguish  knotoids in the same surface with the same crossing number. A knot-type knotoid is given the Rolfsen notation of its associated knot.

\begin{definition}
A {\it multi-linkoid diagram} in a closed surface $\Sigma$ is a generic immersion of a number of unit intervals $[0,1]$ and circles $S^1$ into $\Sigma$. The images of the points 0 and 1 of the unit intervals are pairwise disjoint. Crossing information is given at every transverse double point. 
\end{definition}

\begin{definition}

Consider the equivalence relation on multi-linkoids diagrams generated by Reidemeister moves away from the endpoints and planar isotopy.  The equivalence classes of this relation are called {\it multi-linkoids}.

\end{definition}

\section{Shadow Quandle Colorings of Knots and Knotoids}
\label{sec3}

\begin{definition}
A {\it quandle}  is a set $X$ with a binary operation $(x,y)\mapsto x^y$ such that 

\begin{enumerate}

\item[(i)] for any $x\in X$, it holds that $x^x=x$,
\item[(ii)] for any $x,y\in X$, there exists a unique $z\in X$ such that $z^y=x$, and 
\item[(iii)] for any $x,y,z\in X$, it holds that $(x^y)^z=(x^z)^{(y^z)}$.

\end{enumerate}

\label{def:quandle}
\end{definition}

\begin{example}

The {\it trivial quandle} $E_n$ of order $n$ is determined by the relation $x^y=x$ for all $x,y\in E_n$.

\end{example}

\begin{example}

The {\it dihedral quandle}  $R_n=\mathbb{Z}/n\mathbb{Z}$ of order $n$ is defined by $x^y=2y-x$. The quandle operation can be calculated through the symmetries of regular $n$-gons. Let $R_n$ be the vertices of the regular $n$-gon cyclicly ordered clockwise or counter-clockwise. For two points $x,y\in R_n$, the product $x^y$ is the reflection of $x$ across the line containing $y$ and the center of the $n$-gon. See Figure \ref{fig:dihedral}.

\end{example}

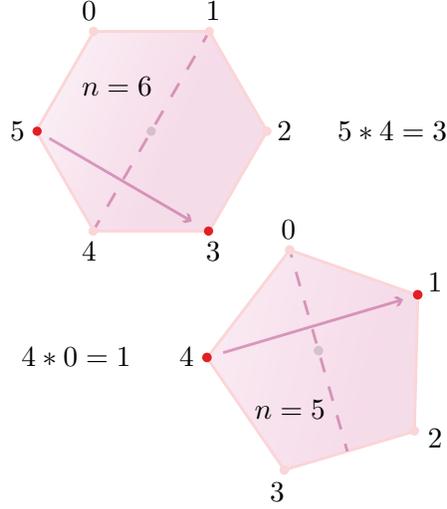
\begin{figure}

\begin{overpic}[unit=1mm,scale=.27]{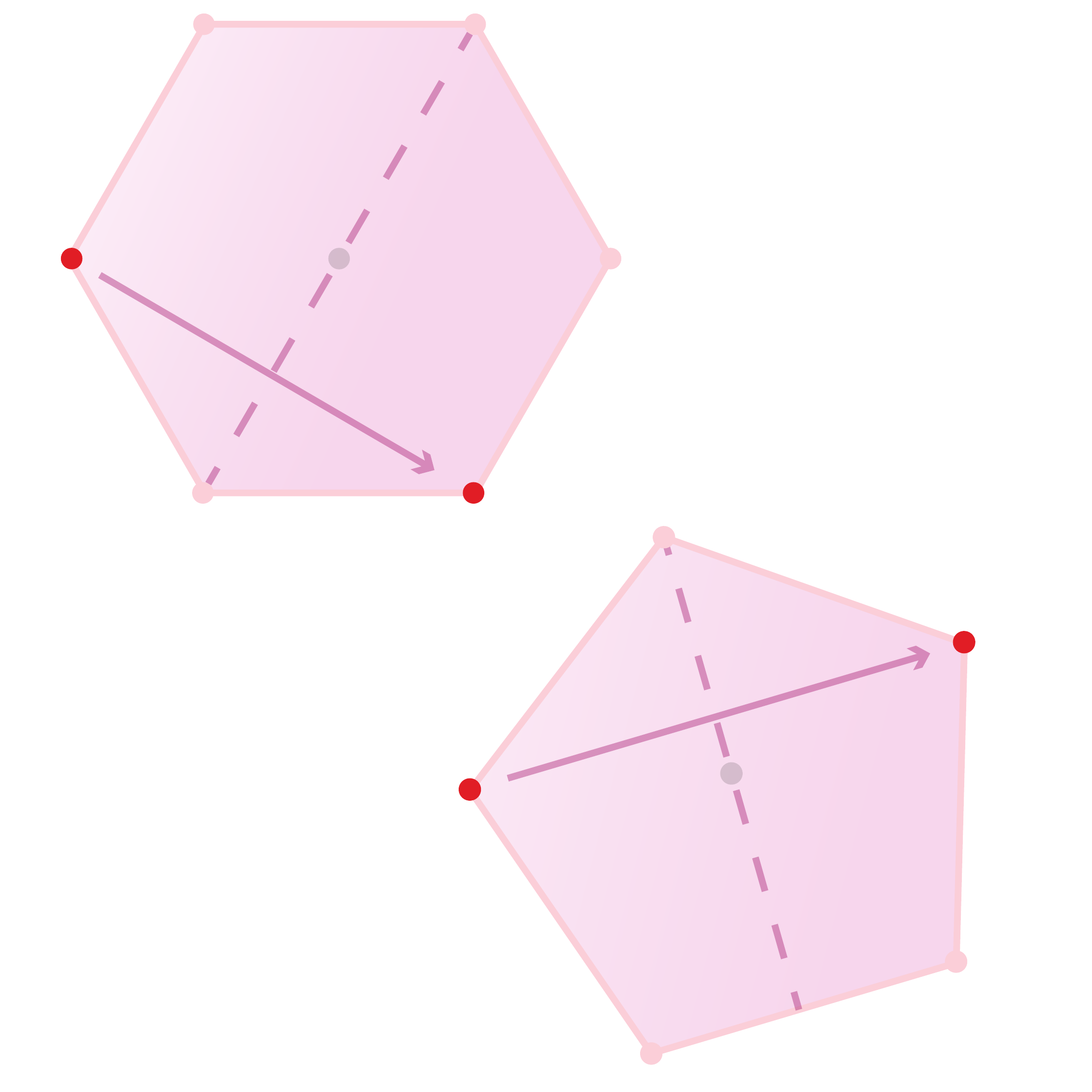}\put(10,62){0}\put(26.5,62){1} \put(36,46.1){2} \put(.5,46.1){5} \put(10,30){4} \put(26.5,30){3}  \put(44,46.1){$5*4=3$} 
\put(10,52){$n=6$}
\put(36.5,33){0} \put(56,26){1} \put(56,5.2){2} \put(35,-2){3} \put(23,16){4} \put(2,16){$4*0=1$}
\put(33,9){$n=5$}
\end{overpic}

\caption{The dihedral quandle operation determined by the symmetries of a regular $n$-gon.}
\label{fig:dihedral}
\end{figure}

\begin{example}

There are only three quandles of order 3: $E_3$, $R_3$, and $P_3$. The multiplication table of $P_3$ is given in Table $\ref{tab:2}$.
\begin{table}[ht]
\centering
\begin{tabular}{c|c c c  }

$P_3$ & 0&1&2 \\\hline 
 0 & 0 &0&0\\ 
1 & 2&1&1\\
2 & 1 &2&2\\

\end{tabular}
\vspace{4mm}
\caption{Multiplication table of $P_3$.}
\label{tab:2}
\end{table}

\end{example}
\
\begin{definition} For a quandle $X$, a {\it quandle coloring} or {\it $X$-coloring} of an oriented knot or knotoid diagram is an assignment of $X$ elements called {\it colors} to each arc such that the assignment satisfies the relation of Figure \ref{fig:colorlink} at each crossing. \end{definition}

The co-orientation normal is a more convenient orientation marking for verifying this relation. Given an oriented arc, the conventional normal vector direction is shown in Figure \ref{fig:colorlink}.

\begin{figure}[ht]

\begin{overpic}[unit=.434mm,scale=.8]{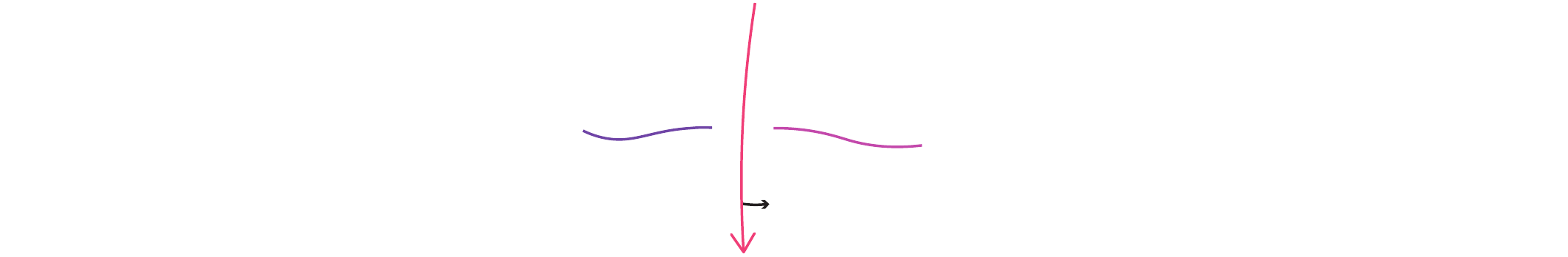}\put(163,35){$x$}
\put(182,47){$y$}\put(212,34){$x^y$}
\end{overpic}

\caption{Quandle coloring condition around a crossing.}\label{fig:colorlink}
\end{figure}

\begin{definition}
For a quandle $X$, an {\it X-set} is a set $Y$ with a map $*:Y\times X\to Y; (y,x)\mapsto y*x$ satisfying the following conditions:

\begin{enumerate}
\item[(i)] For any $x\in X$, $Y\to Y; y\mapsto y*x$ is a bijection.

\item[(ii)] For any $y\in Y$ and any $a,b\in X, (y*a)*b=(y*b)*(a*b)$.

\end{enumerate}

\end{definition}

Let $D$ be a diagram of an oriented knot. Since a knot diagram is a 4-valent planar graph, the complement $\mathbb{R}^2\setminus D $ (or $S^2\setminus D$) is a collection of connected components called {\it regions}. The set of regions of $D$ is denoted Region($D$). With the addition of crossing information, the set of $D$'s arcs is denoted Arc($D$).

\begin{definition} Let $X$ be a quandle and $Y$ and $X$-set. For any oriented knot diagram $D$, an {\it (X,Y)-coloring} of $D$ is a map $c: \text{Arc($D$)}\cup \text{Region($D$)}\to X\cup Y$ satisfying the following conditions:

\begin{enumerate}
\item[(i)] $c(\text{Arc($D$)})\subset X$ and $c(\text{Region($D$)})\subset Y$.

\item[(ii)] The restriction of $c\big|_\text{Arc($D$)}:\text{Arc($D$)}\to X$ is an $X$-coloring of $D$.

\item[(iii)] Let $r_1$ and $r_2$ be adjacent regions of $D$ along an arc $a\in\text{Arc($D$)}$. If the co-orientation normal of $a$ points from $r_1$ to $r_2$, then $c(r_1)^{c(a)}=c(r_2)$. See Figure \ref{fig:colorregion}.

\end{enumerate}

\end{definition}

\begin{figure}[ht]

\begin{overpic}[unit=.434mm,scale=.8]{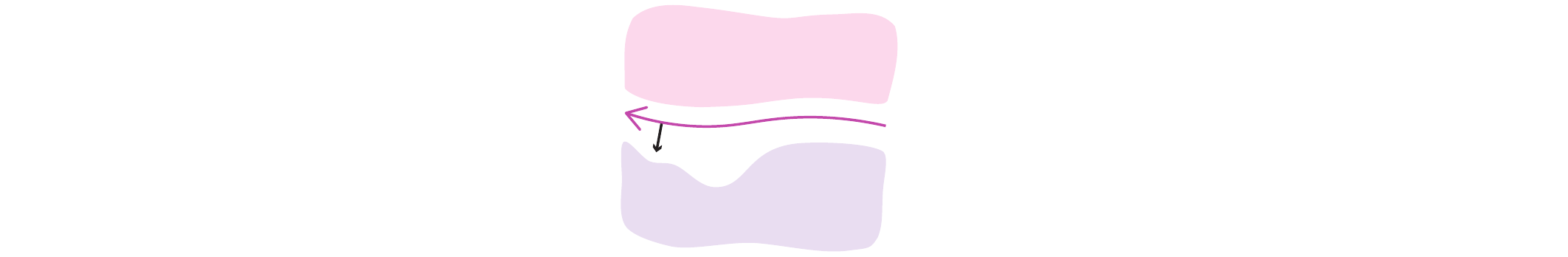}\put(198,13){$c(r_2)$}\put(174,24){$c(a)$} \put(186,48){$c(r_1)$}
\end{overpic}

\caption{Coloring condition of region separated by an arc, $c(r_1)^{c(a)}=c(r_2)$.}

\label{fig:colorregion}
\end{figure}

For a knotoid diagram $D\subset\Sigma$, $\Sigma\setminus D$ is a collection of connected components called {\it regions}, Region($D$). With crossing information, $D$ is a collection of arcs Arc($D$); two arcs are distinct since they contain the head and leg.

\begin{definition}
Let $X$ be a quandle and $Y$ an $X$-set. For any oriented knotoid diagram $D$, an {\it (X,Y)-coloring} of $D$ is a map $c: \text{Arc($D$)}\cup \text{Region($D$)}\to X\cup Y$ satisfying the following conditions:

\begin{enumerate}
\item[(i)] $c(\text{Arc($D$)})\subset X$ and $c(\text{Region($D$)})\subset Y$.

\item[(ii)] The restriction of $c\big|_\text{Arc($D$)}:\text{Arc($D$)}\to X$ is an $X$-coloring of $D$.

\item[(iii)] Let $r_1$ and $r_2$ be adjacent regions of $D$ along an arc $a\in\text{Arc($D$)}$ that does not meet the head nor leg. If the co-orientation normal of $a$ points from $r_1$ to $r_2$, then $c(r_1)^{c(a)}=c(r_2)$. See Figure \ref{fig:colorregion}.

\item[(iv)] Let $r$ be the region of $D$ adjacent to an arc $a\in\text{Arc($D$)}$ that meets the head or leg.  Regardless of orientation, $c(r)^{c(a)}=c(r)$. See Figure \ref{fig:colorregion2}.

\end{enumerate}

\end{definition}

\begin{figure}[ht]

\begin{overpic}[unit=.434mm,scale=.8]{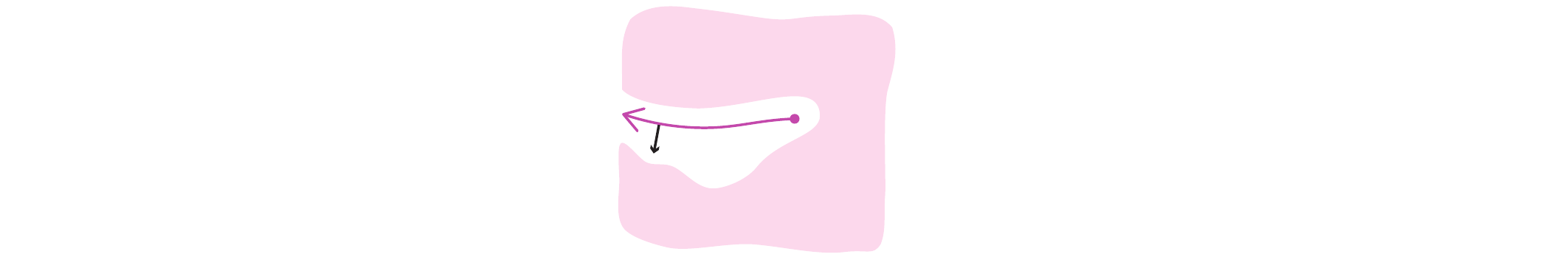}\put(174,24){$c(a)$} \put(186,48){$c(r)$}
\end{overpic}

\caption{Coloring condition around an endpoint arc, $c(r)^{c(a)}=c(r)$.}

\label{fig:colorregion2}
\end{figure}

 An $(X,Y)$-coloring is also called a {\it shadow quandle coloring}. 
 
 \begin{definition} The set of $(X,Y)$-colorings of a knot or knotoid diagram $D$ is denoted {\it $Col_{(X,Y)}(D)$.} The cardinal number of $Col_{(X,Y)}(D)$ is denoted by $col_{(X,Y)}(D)$ and called the {\it (X,Y)-coloring number} or {\it shadow coloring number} by $(X,Y)$. 
\end{definition}

\begin{proposition}[Proposition 9.4.3 of \cite{kamada2017surface}]

If two diagrams $D$ and $D'$ present equivalent oriented knots, then there is a bijection between $Col_{(X,Y)}(D)$ and $Col_{(X,Y)}(D')$. The shadow coloring number $col_{(X,Y)}(D)$ is an invariant of an oriented knot.
\label{prop:knot}
\end{proposition}

\noindent After a Reidemeister move on a $(X,Y)$-colored knot diagram there is a unique coloring extending the original. This establishes the bijection of the proposition.

\begin{proposition}

If two orientated knotoid diagrams $D$ and $D'$ present equivalent knotoids, then there is a bijection between $Col_{(X,Y)}(D)$ and $Col_{(X,Y)}(D')$. The shadow coloring number $col_{(X,Y)}(D)$ is an invariant of a knotoid.

\end{proposition}

\begin{proof}

This follows from the proof of the knot case, Proposition \ref{prop:knot}, seen in \cites{kamada2017surface,CKS,shadowquandle,nosaka2017quandles}. After each Reidemeister move, there is a unique choice of color for any newly formed arc or region. For example, Figure \ref{fig:r3} shows that after a Reidemeister III move there is a new region and arc created. Since a color $z$ assigned to the new region or arc must satisfy the relation $z^x=y$ for known colors $x$ and $y$, there exists a unique $z$ by the condition (ii) of Definition \ref{def:quandle}. Similar arguments work in the simpler cases of Reidemeister I and II moves. Since Reidemeister moves on knotoid diagram occur away from the head and leg, these unique extension arguments apply for both knots and knotoids. 

\end{proof}

\begin{figure}[ht]

\includegraphics[scale=.8]{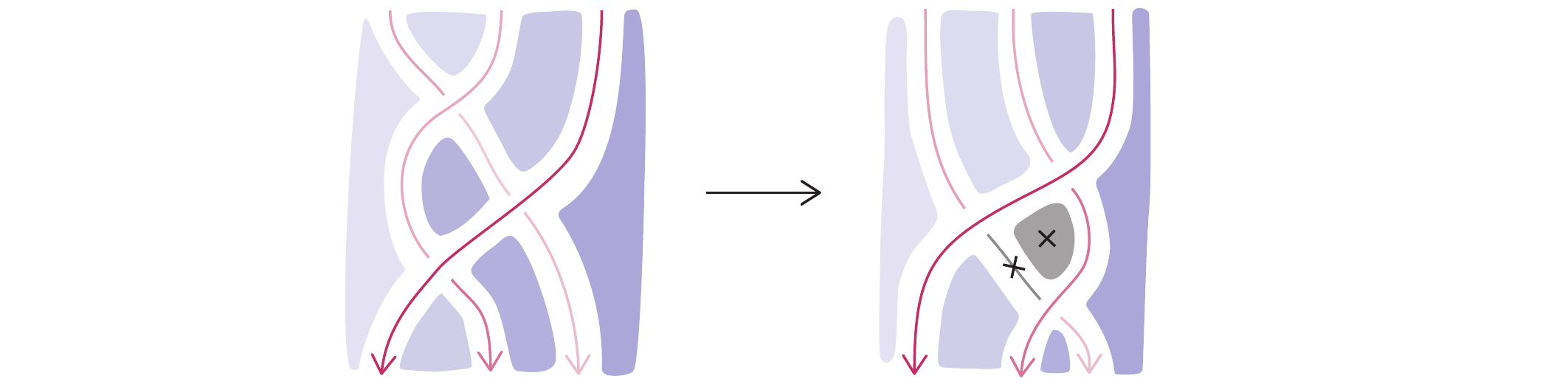}

\caption{Marked arc and region created after a Reidemeister III move.}

\label{fig:r3}
\end{figure}

\begin{example}
For a quandle $X$, the number of $X$-colorings of a knotoid diagram $D$ is an invariant denoted $col_X(D)$. The proof of this follows from the proof of the shadow coloring number's invariance. It's inferred from Figure \ref{fig:2_4} that  $col_{R_n}(2_4)=col_{R_n}(3_2)=n$ while $col_{R_3}(\text{mir}(2_4))>3$ and $col_{R_5}(\text{mir}(3_2))>5$, i.e. $2_4$ and $3_2$ are chiral.

\end{example}

\begin{figure}[ht]
\begin{overpic}[unit=.434mm,scale=.8]{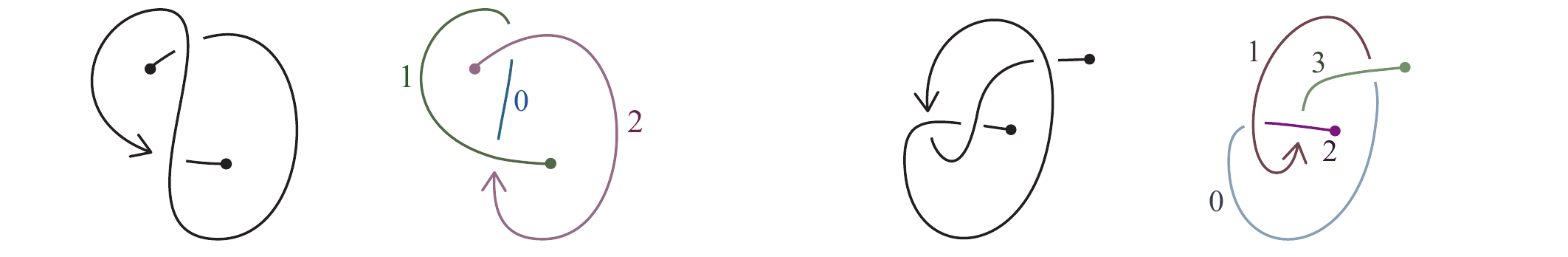}\put(77,4){$2_4$}\put(159,4){mir($2_4$)}\put(265,4){$3_2$}\put(348,4){mir($3_2$)}
\end{overpic}
\caption{A 3-coloring of the planar knotoid mir($2_4$) and a 5-coloring of the classical knotoid mir($3_2$).}
\label{fig:2_4}
\end{figure}

\begin{figure}

\begin{overpic}[unit=.434mm,scale=.8]{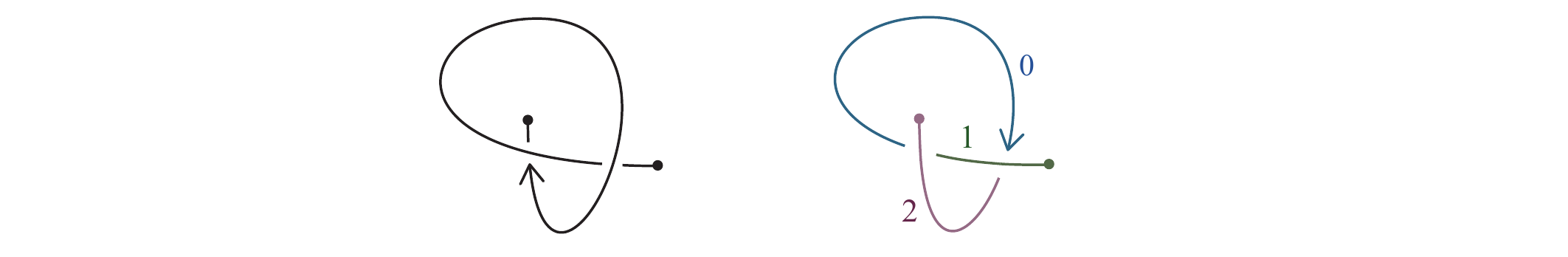}\put(163,4){$2_1$}\put(263,4){mir($2_1$)}

\end{overpic}

\caption{Knotoid diagrams of $2_1$ and mir($2_1$).}
\label{fig:2_1}
\end{figure}

\begin{example}

Consider the knotoid diagram of $2_1$ in Figure \ref{fig:2_1}. This diagram is $(R_3,R_3)$-colorable. Since $2y-x=x$ implies $y=x$ over $\mathbb{Z}/3\mathbb{Z}$, a choice of the infinite region's color determines the color of the head arc. The head arc and infinite region's color uniquely determines a shadow coloring, a trivial one. Therefore, $col_{(R_3,R_3)}(2_1)=3$. Now consider the diagram of mir($2_1$) in Figure \ref{fig:2_1}. This diagram is also $(R_3,R_3)$-colorable.  As before, a choice of the infinite region's color determines the color of the head arc. A second choice of color for either remaining arc uniquely determines a $(R_3,R_3)$-coloring. Thus, $col_{(R_3,R_3)}(\text{mir}(2_1))=9$. Since $col_{(R_3,R_3)}(2_1)\neq col_{(R_3,R_3)}(\text{mir}(2_1))$, the classical knotoid $2_1$ is chiral.
\label{ex:2_1}

\end{example}

\begin{theorem}

The shadow coloring number distinguishes infinitely many knotoids from their mirrors.
\label{thm:colornumber}
\end{theorem}

\begin{proof}

\begin{figure}[ht]
\begin{overpic}[unit=.434mm,scale=.8]{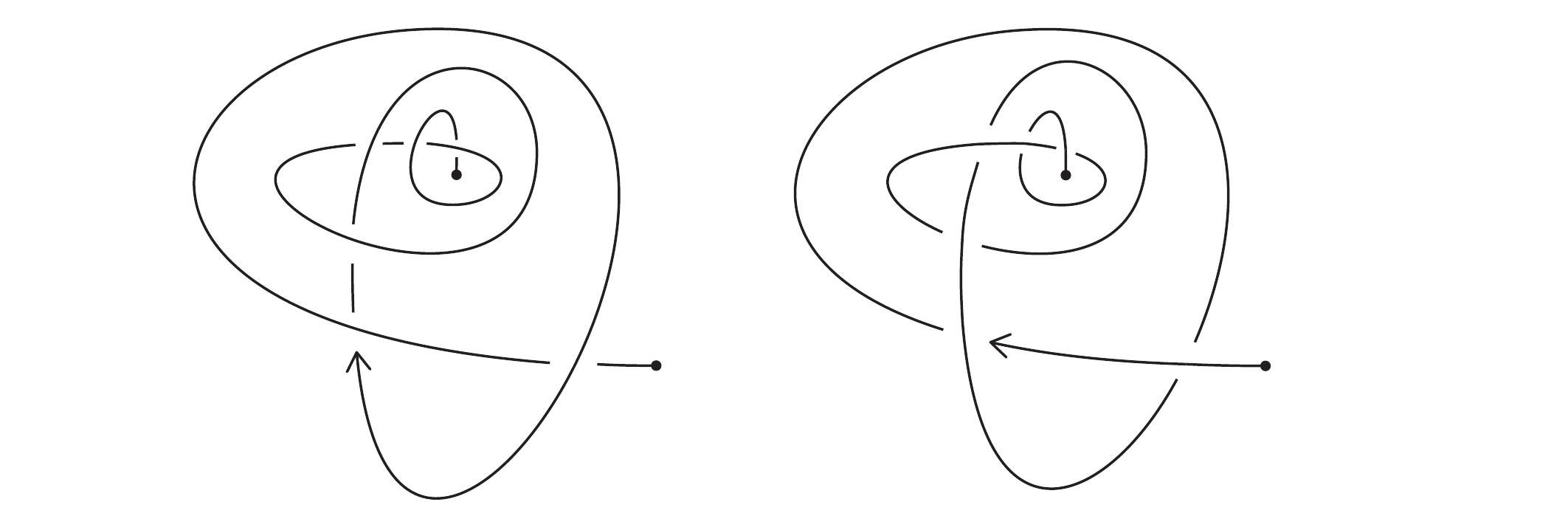}\put(55,10){$\prod_{i=1}^32_1$}\put(190,10){$\prod_{i=1}^3\text{mir}(2_1)$}
\end{overpic}

\caption{Product diagrams of $2_1$ and mir($2_1$) powers.}
\label{fig:product21}
\end{figure}

Consider the knotoid diagrams of the products $\prod_{i=1}^n 2_1$ and $\prod_{i=1}^n \text{mir}(2_1)$ in Figure \ref{fig:product21}. A recursion of $2_1$ and mir($2_1$)'s $(R_3,R_3)$-coloring choices determines the number of $(R_3,R_3)$-colorings of the products. In the case of $\prod_{i=1}^n 2_1$, the infinite region's color uniquely determines a trivial coloring as in Example \ref{ex:2_1}. Therefore,

\[ 
col_{(R_3,R_3)}\left(\prod_{i=1}^n 2_1\right)=3.
\]

In the case of $\prod_{i=1}^n \text{mir}(2_1)$, the infinite region's color only determines the head arc's color. There are 3 choices of colors for the remaining arcs of the first mir($2_1$) factor. This choice also colors the first arc of the second factor. There are 3 choices of colors for the remaining arcs of the second mir($2_1$) factor. This inductively continues for each factor. Once the arcs are colored, the infinite region's color determines the color of the remaining regions. Thus,

\[ 
col_{(R_3,R_3)}\left(\prod_{i=1}^n \text{mir}(2_1)\right)=3^{n+1}.
\]

\end{proof}

\begin{remark}
Shadow quandle colorings and the shadow coloring number naturally extend to links and multi-linkoids.
\end{remark}

\section{The Shadow Quandle Cocycle Invariant of Knots and Knotoids}
\label{sec4}

Let $X$ be a quandle and $Y$ an $X$-set. Let $C_n^R(X)_Y$ denote the free abelian group generated by the elements of the Cartesian product $Y\times X^n$. When $n=0$, $C_0^R(X)_Y=\mathbb{Z}[Y]$. If $n<0$, then $C_n^R(X)_Y=0$. Define a homomorphism $\partial_n:C_n^R(X)_Y\to C_{n-1}^R(X)_Y$ by \begin{align*} \partial_n(y,x_1,\dots,x_n)=\sum_{i=1}^n & (-1)^i(y,x_1,\dots,\widehat x_i,\dots,x_n)\\ -& \sum_{i=1}^n(-1)^i(y*{x_i},x_1^{x_i},\dots,x_{i-1}^{x_i},\widehat x_i, x_{i+1}, \dots, x_n). \end{align*}

\noindent Then $(C_n^R(X)_Y,\partial_n)$ is a chain complex \cites{kamada2017surface,nosaka2017quandles,CKS,elhamdadi2015quandles}. Let $C_n^D(X)_Y$ denote the subgroup of $C_n^R(X)_Y$ generated by all elements $(y,x_1,\dots,x_n)$ such that $x_i=x_{i+1}$ for some $i$. For $n<2$, put $C_n^D(X)_Y=0$. Then $(C_n^D(X),\partial_n)$ is a chain sub-complex of $(C_n^R(X)_Y,\partial_n)$. Define $C_n^Q(X)_Y:=C_n^R(X)_Y/C_n^D(X)_Y$ so that $(C_n^Q(X),\partial_n)_Y$ is a quotient complex. The homology groups of the chain complex $(C_n^Q(X),\partial_n)_Y$ are the {\it quandle homology groups} of $X$ with an $X$-set $Y$.

\begin{definition} Let $D$ be a diagram of an oriented knot or knotoid with co-orientation normals assigned to its arcs. Let $v$ be a crossing of $D$. Among the four regions around $v$ there is one such that both normals of its adjacent arcs point from the region. This region is called the {\it specified region} of $v$. A star will be placed in the corner of each crossing's specified region. The under arc adjacent to the specified region is called the {\it specified arc}.

\end{definition}

\begin{figure}[ht]

\begin{overpic}[unit=.434mm,scale=.8]{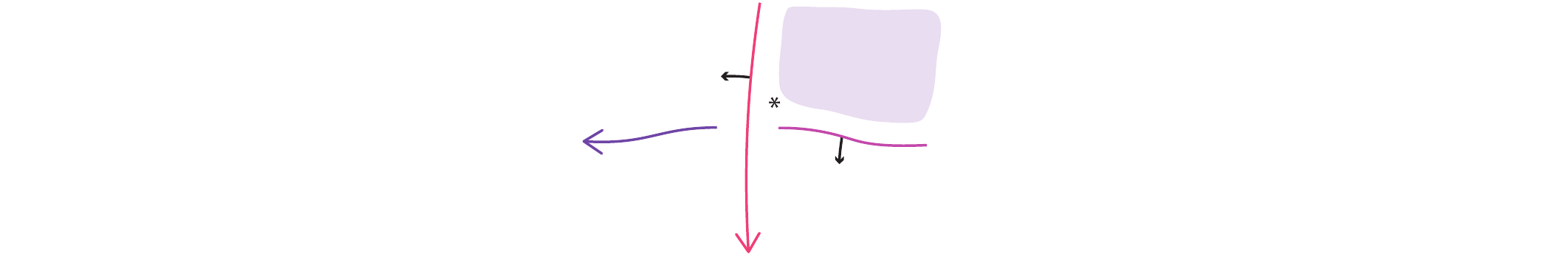}\put(221,21){$x_1$}\put(180,56){$x_2$}\put(217,48){$y$}
\end{overpic}

\label{fig:specified}
\caption{Starred specified region of a colored crossing.}
\end{figure}

\begin{definition}
Let $X$ be a quandle, $Y$ and $X$-set, $A$ an abelian group, $f\in$ Hom($C_2^R(X)_Y,A$) a 2-cochain, and $c$ an $(X,Y)$-coloring of an oriented knot or knotoid diagram $D$. For a crossing $v$ of $D$, let $W_f(v,c)\in A$ be defined by \[ W_f(v,c)= \begin{cases} 
      f(y,x_1,x_2) & \ \text{if  $v$ is positive} \\
          -f(y,x_1,x_2) & \text{if $v$ is negative}, \\
   \end{cases}\]
where $(y,x_1,x_2)$ are the colors of the specified region, the specified under-arc, and the over-arc. The value $W_f(v,c)$ is called the {\it local weight} at $v$, and $(y,x_1,x_2)$ is the {\it color} of the crossing. The {\it $f$-weight} of $D$ is the total of the local weights and is denoted by \[ W_f(D,c):=\sum_v W_f(v,c).\]

\end{definition}

\begin{definition} If $f$ satisfies the following two conditions, then the weight $W_f(D,c)$ is preserved under Reidemeister moves: 

\begin{enumerate}
\item[(i)] For any $y\in Y$ and any $x_1,x_2,x_3\in X,$ \begin{align*}
&-f(y,x_2,x_3)+f(y*x_1,x_2,x_3)+f(y,x_1,x_3) \\ 
&-f(y*x_2,x_1^{x_2},x_3)-f(y,x_1,x_2)+f(y*x_3,x_1^{x_3},x_2^{x_3})=0
\end{align*}

\item[(ii)] For any $y\in Y$ and $x\in X$, $f(y,x,x)=0$.

\end{enumerate}
In fact, these conditions are necessary and sufficient for a 2-cochain $f\in C_2^R(X)_Y\to A$ to be a 2-cocycle of the complex Hom($C_*^Q(X)_Y,A$) \cites{kamada2017surface,nosaka2017quandles,CKS,elhamdadi2015quandles}. Such a homomorphism $f:Y\times X^2\to A$ satisfying the two conditions is called a {\it quandle 2-cocycle} of $(X,Y)$.

\end{definition}

\begin{theorem}[Theorem 9.4.5 of \cite{kamada2017surface}]

Let $f$ be a quandle 2-cocycle of $(X,Y)$ and $D$ an oriented diagram of a knot $K$. Define $\Phi_f(D):=\{W_f(D,c)| c\in Col_{(X,Y)}(D)\}$ as a multiset. It is an invariant of $K$, and if $f$ and $f'$ are cohomologous, then $\Phi_f(K)=\Phi_{f'}(K)$.
\label{thm:knot}
\end{theorem}

For a subset $Y_0\subset Y$, let $Col_{(X,Y)}(D)_{Y_0}$ denote the subset of $Col_{(X,Y)}(D)$ consisting of $(X,Y)$-colorings such that the infinite region is colored by elements of $Y_0$. Then \[ \Phi_f(D)_{Y_0}:=\{W_f(D,c)| c\in Col_{(X,Y)}(D)_{Y_0}\}\] is also an invariant.

\begin{theorem}

Let $f$ be a quandle 2-cocycle of $(X,Y)$ and $D$ an oriented diagram of a knotoid $\kappa$. Define $\Phi_f(D):=\{W_f(D,c)| c\in Col_{(X,Y)}(D)_{Y_0}\}$ as a multiset. It is an invariant of $\kappa$, and if $f$ and $f'$ are cohomologous, then $\Phi_f(\kappa)=\Phi_{f'}(\kappa)$.

\end{theorem}

\begin{proof}

This follows from the proof of Theorem \ref{thm:knot} seen in \cites{kamada2017surface,CKS,shadowquandle,nosaka2017quandles,elhamdadi2015quandles}. The proof of the former statement follows from analyzing the local weights of the crossings around a Reidemeister move. The $f$-weight does not change after a Reidemeister move. The proof of invariance among homology class representatives is a geometric consequence of examining the 1-cochain whose boundary is the difference $f-f'$.

\end{proof}

\begin{definition} Let $D$ be a diagram of an oriented knot or knotoid. Let $A$ be an abelian group written multiplicatively and $X$ a finite quandle. For a quandle 2-cocycle $f$ of $(X,Y)$,

\[ \Phi_f'(D):=\sum_{c\in Col_{(X,Y)}(D)}\prod_v W_f(v,c) \]
is an element of the group ring $\mathbb{Z}[A]$ and is an invariant of an oriented knot. This invariant is called the {\it shadow quandle cocycle invariant}. For a subset $Y_0\subset Y$, the invariant $\Phi_f'(D)_{Y_0}$ is also defined by restricting the colorings to $Col_{(X,Y)}(D)_{Y_0}$.

\end{definition}

When $X=Y$, a {\it quandle 3-cocycle} of $X$ is a quandle 2-cocycle of $(X,Y)$.

\begin{definition} Let $X$ be a quandle, $A$ an abelian group, and $\mathbb{Z}(X^3)$ the free $\mathbb{Z}$-module generated by the elements of $X^3=X\times X\times X$. A homomorphism $f: \mathbb{Z}(X^3)\to A$ is a {\it quandle 3-cocycle} of $X$ if the following conditions are satisfied: 

\begin{enumerate}
\item[(i)] For any $x_0,x_1,x_2,x_3\in X,$ \begin{align*}
&-f(x_0,x_2,x_3)+f(x_0^{x_1},x_2,x_3)+f(x_0,x_1,x_3) \\ 
&-f(x_0^{x_2},x_1^{x_2},x_3)-f(x_0,x_1,x_2)+f(x_0^{x_3},x_1^{x_3},x_2^{x_3})=0
\end{align*}

\item[(ii)] For any $x,y\in X$, $f(y,x,x)=f(x,x,y)=0$.

\end{enumerate}

\end{definition}

 There is an associated chain and cochain complex of $X$. Quandle 3-cocycles are cocycles of this cochain complex and represent cohomology classes of $H^3_{Q}(X,A)$,  see \cites{cocycle,CKS,Carter1999ComputationsOQ,kamada2017surface,nosaka2017quandles}. These complexes are analogous to the complexes defined in the beginning of this section with the $X$-set $Y$ omitted.

\begin{theorem}[Mochizuki '10 \cite{mochi}]

Let $m>0$ be any odd integer, and let $p$ be any prime. Then, $H_Q^3(R_m,\mathbb{Z}_p)\cong \mathbb{Z}_p$ if $m$ is divisible by $p$ and $H_Q^3(R_m,\mathbb{Z}_p)\cong 0$ otherwise.

\end{theorem}

\begin{example}

The cocycle $\theta_3:R_3\times R_3 \times R_3 \to \mathbb{Z}_3=\langle u |u^3=1\rangle$ given by  \[ \theta_3(x,y,z)=u^{(x-y)(y-z)z(x+z)}\]  generates $H_Q^3(R_3,\mathbb{Z}_3)$. This cocycle has been used to distinguish the trefoil from its mirror \cite{rourke}, distinguish the 2-twist-spun trefoil from its mirror \cites{cocycle,rourke}, and  calculate the triple point number of the 2-twist-spun trefoil \cite{satoh1}.

% and the cocycle $\theta_5:R_5\times R_5\times R_5\to \mathbb{Z}_5=\langle t| t^5=1\rangle$ given by  \[ \theta_5(x,y,z)=t^{(y-x)(3z^5+2y^3z^2+2y^2z^3+3yz^4+3xy^4+3xy^3z+xy^2z^2+2xyz^3+4xz^4)} \] generates $H_Q^3(R_5,\mathbb{Z}_5)$. 

\end{example}

\begin{example}

Let $\phi:P_3\times P_3 \times P_3\to \mathbb{Z}_2\oplus \mathbb{Z}$ be given by \[ \phi (a,b,c) = \begin{cases} 
    1\oplus0 & \ (a,b,c)=(0,1,0),(0,2,0) \\
     0\oplus1 & \ (a,b,c)=(1,0,2),(2,0,1) \\
         0\oplus-1 & \ (a,b,c)=(1,0,1),(2,0,2) \\
             0\oplus0 & \ \text{otherwise.}
   \end{cases}\]
   
  \noindent The linear extension $\phi: \mathbb{Z}(P_3\times P_3 \times P_3)\to\mathbb{Z}_2\oplus\mathbb{Z}$ is quandle 3-cocycle of $P_3$. Oshiro used this linear extension as a symmetric quandle 3-cocycle to compute the triple point numbers of infinitely many non-orientable surface-links \cite{oshiro}.

\label{ex:p3}

\end{example}

\begin{remark}
The shadow quandle cocycle invariant naturally extends to links and multi-linkoids.
\end{remark}

\section{Applications of the Shadow Quandle Cocycle Invariant of Knotoids}
\label{sec5}

\begin{theorem}

The knotoid $3_1$ is chiral.
\label{thm:trefoil}
\end{theorem}

\begin{proof}

 The knotoid diagrams of $3_1$ and mir($3_1$) in Figure \ref{fig:regioncolored} are $(R_3,R_3)$-colored. Since $2y-x=x$ implies $y= x$ over $\mathbb{Z}/3\mathbb{Z}$, the infinite region's color determines the colors of the head and leg arcs. With $Y_0=\{0\}$, the head and leg arcs must also be colored 0. Then, a choice of color for the marked arcs uniquely determines a $(R_3,R_3)$-coloring of the diagrams. Let $c_i$ be the coloring determined by coloring the marked arc $i$. Therefore, $col_{(R_3,R_3)}(3_1)_{Y_0}=col_{(R_3,R_3)}(\text{mir}(3_1))_{Y_0}=3$ with one trivial coloring and two non-trivial. The two non-trivial colorings differ by the permutation $0\mapsto 0$, $1\mapsto2$, $2\mapsto1$ and are illustrated in Figure \ref{fig:regioncolored}.

\begin{figure}[ht]

\begin{overpic}[unit=.434mm,scale=.8]{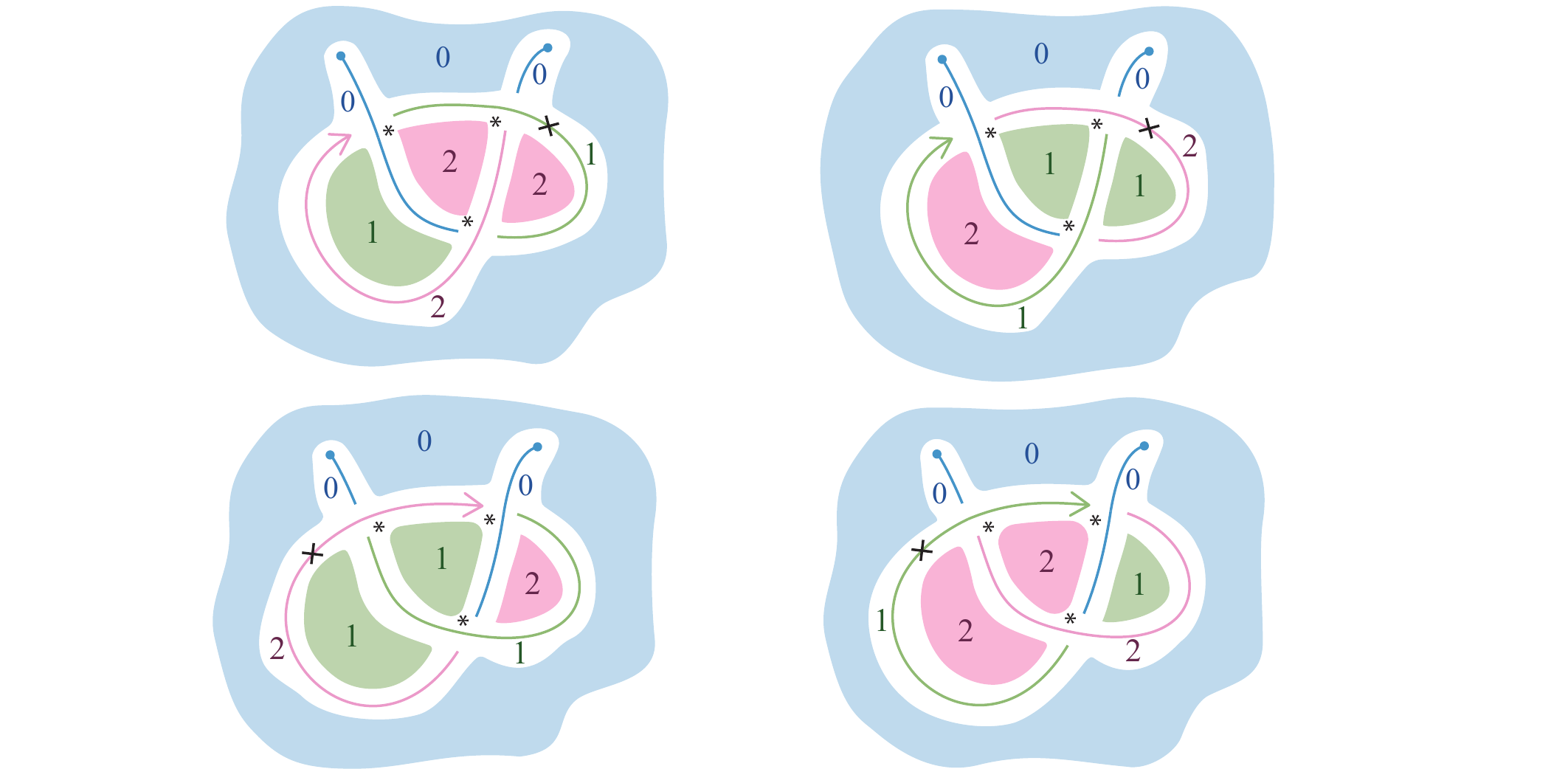}\put(156,110){$3_1$}\put(150,8){mir($3_1$)}

\put(309,110){$3_1$}\put(309,8){mir($3_1$)}
\end{overpic}

\caption{Non-trivial $(R_3,R_3)$-colorings of $3_1$ and mir($3_1$) with $Y_0=\{0\}$.}
\label{fig:regioncolored}
\end{figure}
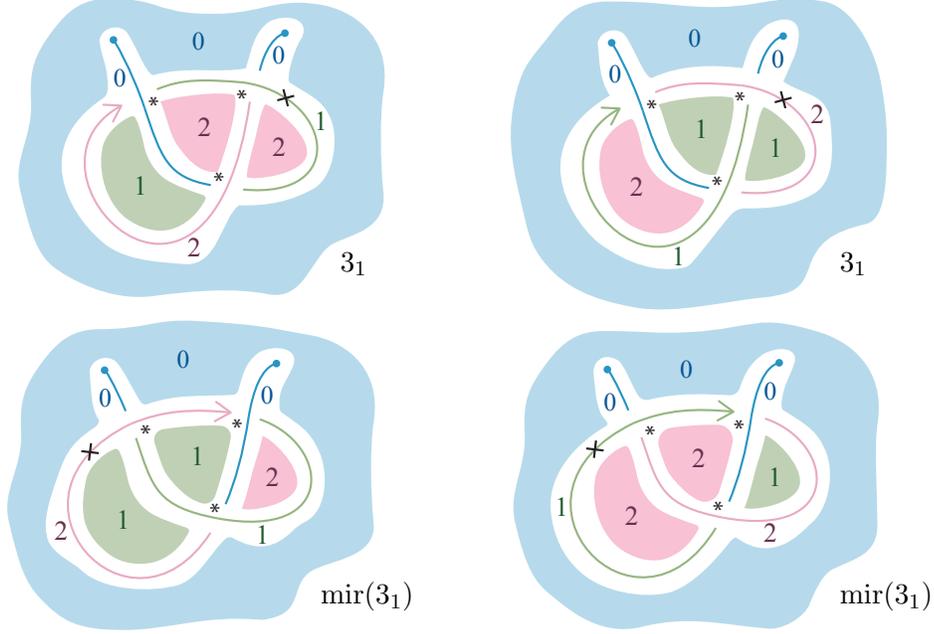

Recall that $\theta_3:R_3\times R_3 \times R_3 \to \mathbb{Z}_3=\langle u |u^3=1\rangle$ given by  \[ \theta_3(x,y,z)=u^{(x-y)(y-z)z(x+z)}\]  generates $H_Q^3(R_3,\mathbb{Z}_3)$.  For each coloring, Table \ref{tab:color1} tabulates the 3-chains $\in C_3^Q(R_3)$ representing the sum of crossing colors and the $\theta_3$-weights of each coloring.

\begin{table}[ht]
\setlength{\tabcolsep}{10pt} % Default value: 6pt
\renewcommand{\arraystretch}{1.5}

\begin{tabular}{|c||c |c |}

\hline 
$c_i$ & Crossing Color Sum of $3_1$ & $W_{\theta_3}(3_1,c_i)$     \\ \hline

$c_0$ & $-(0,0,0)-(0,0,0)-(0,0,0)$& 1

 \\ \hline

$c_1$ & $-(2,2,1)-(2,0,2)-(2,1,0)$& $u^2$

\\\hline

$c_2$ & $-(1,1,2)-(1,0,1)-(1,2,0)$ & $u^ 2$

\\\hline

\end{tabular}

\vspace{5mm}

\begin{tabular}{|c||c |c |c |c| c| c |c |c |c| c| c| c| c| c| c| c| c| c| c| c| c| }

\hline 
$c_i$ & Crossing Color Sum of mir($3_1$)& $W_{\theta_3}(\text{mir}(3_1),c_i)$    \\ \hline

$c_0$ &$(0,0,0)+(0,0,0)+(0,0,0)$& 1

 \\ \hline

$c_1$ & $(2,1,0)+(2,0,2)+(2,2,1)$&$u$

\\\hline

$c_2$ &   $(1,2,0)+(1,0,1)+(1,1,2)$&$ u$

\\\hline

\end{tabular}

\vspace{5mm}

\caption{$(R_3,R_3)$-colorings of $3_1$ and mir($3_1$).}\label{tab:color1}
\end{table}

With respect to $(R_3,R_3)$, the shadow quandle cocycle invariants of these knotoids are

 \[\Phi'_{\theta_3}(3_1)_{Y_0}=1+2u^2\in \mathbb{Z}[\mathbb{Z}_3],\] and \[\Phi'_{\theta_3}(\text{mir}(3_1))_{Y_0}=1+2u\in \mathbb{Z}[\mathbb{Z}_3].\]

\end{proof}

\begin{corollary}

 For $Y_0=\{0\}$,

\[col_{(R_3,R_3)}\left(\prod_{i=1}^n \kappa_i\right)_{Y_0}=3^n,\]
where each knotoid $\kappa_i$ is either $3_1$ or mir$(3_1)$.

\end{corollary}

\begin{figure}[ht]

\begin{overpic}[unit=.434mm,scale=.8]{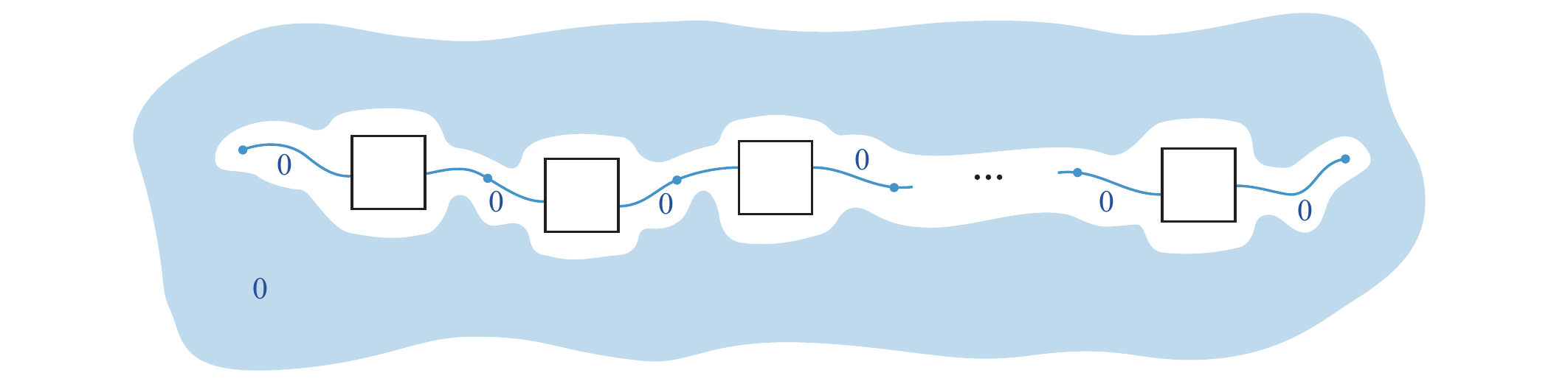}\put(95,52){$\kappa_1$}
\put(144,46){$\kappa_2$}\put(193,51){$\kappa_3$}\put(300,49){$\kappa_n$}

\end{overpic}

\caption{Knotoid product with $3_1$ and mir($3_1$) factors.}
\label{fig:product}
\end{figure}
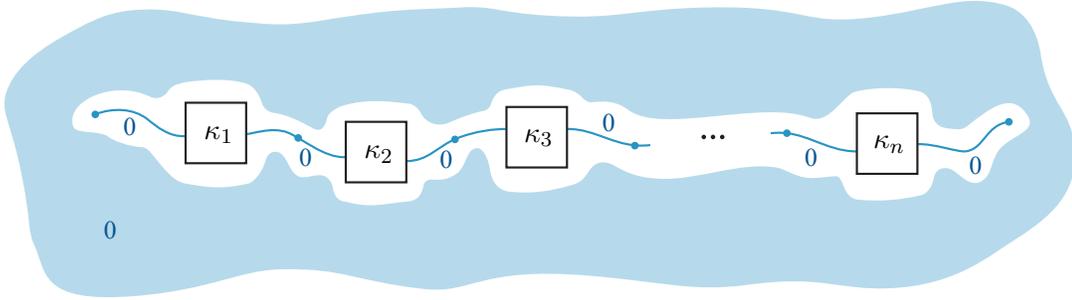
\begin{proof}

Consider the knotoid diagram of $\prod_{i=1}^n \kappa_i$ in Figure \ref{fig:product}. This diagram is seen as the concatenation of $n$ diagrams from Figure \ref{fig:regioncolored} and is $(R_3,R_3)$-colorable. With $Y_0=\{0\}$, there are $n$ independent choices of arc colors to uniquely determine a shadow coloring, the $n$ marked arcs of each $\kappa_i$ component seen in Figure \ref{fig:regioncolored}.

\end{proof}

\begin{theorem}

The shadow quandle cocycle invariant distinguishes infinitely many knot-type knotoids from their mirrors.
\label{thm:cocycle}
\end{theorem}

\begin{figure}[ht]

\begin{overpic}[unit=.434mm,scale=.8]{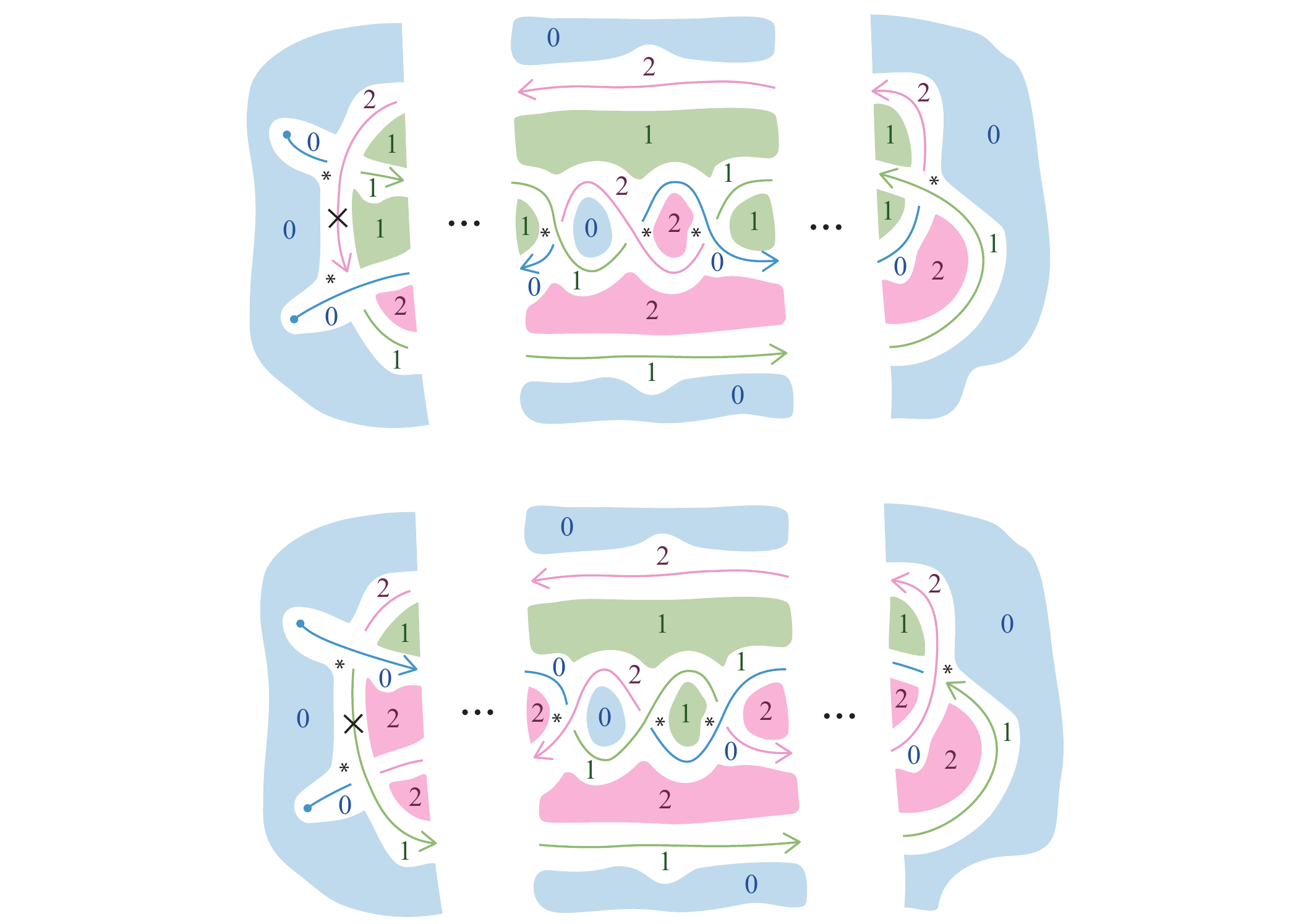}\put(296,158){$T_{3k+1}$}
\put(294,8){mir($T_{3k+1}$)}

\end{overpic}

\caption{The colorings $c_2$ of $T_{3k+1}$ and $c_1$ of mir($T_{3k+1}$).}
\label{fig:twist}
\end{figure}

\begin{proof}

For a non-negative integer $k$, consider the knotoid diagrams $T_{3k+1}$ and mir($T_{3k+1}$) shown in Figure \ref{fig:twist}. These diagrams represent knot-type knotoids and are $(R_3,R_3)$-colorable. Their associated knots are the $3k+1$-twist knot and its mirror. With $Y_0=\{0\}$, the head and leg arcs must be colored 0 and a choice of color for the marked arc uniquely determines a $(R_3,R_3)$-coloring. Therefore, $col_{(R_3,R_3)}(T_{3k+1})_{Y_0}=col_{(R_3,R_3)}(\text{mir}(T_{3k+1}))_{Y_0}=3$. The two non-trivial colorings differ by the permutation $0\mapsto 0$, $1\mapsto2$, $2\mapsto1$, one of which is shown in Figure \ref{fig:twist}. Let $c_i$ denote the coloring determined by coloring the marked arc $i$ in the diagrams of Figure \ref{fig:twist}.

These diagrams are drawn to show the repetition of the colorings every 3 twists. Each diagram will have $k$ copies of the illustrated middle twisted region. Recall that $\theta_3:R_3\times R_3 \times R_3 \to \mathbb{Z}_3=\langle u |u^3=1\rangle$ given by  \[ \theta_3(x,y,z)=u^{(x-y)(y-z)z(x+z)}\]  generates $H_Q^3(R_3,\mathbb{Z}_3)$. 
Suppose that $m\geq0$ is a multiple of 9. Then, the crossing colors of the twisted region in $T_{m+1}$ and mir($T_{m+1}$) come in multiples of 3 and do not contribute to the $\theta_3$-weight of the diagram. Table \ref{tab:color2} tabulates the crossing colors of the 3 crossings not in the $m$-twist region and the $\theta_3$-weights of each coloring. This table shows that \[ \Phi'_{\theta_3}(T_{m+1})_{Y_0}= 1 + 2u^2\in\mathbb{Z}[\mathbb{Z}_3],\] and \[ \Phi'_{\theta_3}(\text{mir($T_{m+1}$)})_{Y_0}=1+2u\in\mathbb{Z}[\mathbb{Z}_3].\]

\begin{table}[ht]
\setlength{\tabcolsep}{10pt} % Default value: 6pt
\renewcommand{\arraystretch}{1.5}

\begin{tabular}{|c||c |c |}

\hline 
$c_i$ & Reduced Crossing Color Sum of $T_{m+1}$  & $W_{\theta_3}(T_{m+1},c_i)$     \\ \hline

$c_0$ &$(0,0,0)+(0,0,0)-(0,0,0)$& 1

 \\ \hline

$c_1$ & $(0,0,1)+(0,1,0)-(0,1,2)$& $ u^2$

\\\hline

$c_2$ &   $(0,0,2)+(0,2,0)-(0,2,1)$& $u^2$

\\\hline

\end{tabular}

\vspace{5mm}

\begin{tabular}{|c||c |c |c |c| c| c |c |c |c| c| c| c| c| c| c| c| c| c| c| c| c| }

\hline 
$c_i$ & Reduced Crossing Color Sum of mir($T_{m+1}$) & $W_{\theta_3}(\text{mir}(T_{m+1}),c_i)$    \\ \hline

$c_0$ &$-(0,0,0)-(0,0,0)+(0,0,0)$& 1

 \\ \hline

$c_1$ & $-(0,1,0)-(0,0,1)+(0,1,2)$& $u$

\\\hline

$c_2$ &   $-(0,2,0)-(0,0,2)+(0,2,1)$& $u$

\\\hline

\end{tabular}

\vspace{5mm}

\caption{$(R_3,R_3)$-colorings of $T_{m+1}$ and mir($T_{m+1}$) with $Y_0=\{0\}$.}\label{tab:color2}
\end{table}

\end{proof}

For non-negative integers $s$ and $t$, let $A_{s,t}$ denote the direct sum of $s$ copies of $\mathbb{Z}_2$ and $t$ copies of $\mathbb{Z}$, $A_{s,t}=(\mathbb{Z}_2)^s\oplus(\mathbb{Z})^t$. Every element of $A_{s,t}$ is of the form $(\alpha_1 \oplus \cdots \oplus \alpha_s) \oplus ( \beta_1\oplus \cdots \oplus \beta_t)$, where $\alpha_i$ is an entry of the $i$th copy of $\mathbb{Z}_2$ and $\beta_j$ is an entry of the $j$th copy of $\mathbb{Z}$. Let $p_i$ and $q_j$ be the elements of $A_{s,t}$ whose entries are all zeros except $\alpha_i=1$ and $\beta_j=1$.

\begin{lemma}[Kamda and Oshiro '09 \cite{oshiro}]

Let $X$ be a quandle, and let $f: \mathbb{Z}(X^3)\to A_{s,t}$ be a 3-cocycle of $X$ such that for any generator $(a,b,c)\in X^3$ of $\mathbb{Z}(X^3)$ it holds that \[ f(a,b,c)\in\{0,p_i,\pm q_j \}.\]
If the $f$-weight $W_f(D,c)$ of a multi-linkoid diagram $D$ with an $(X,X)$-coloring is equal to $(\alpha_1 \oplus \cdots \oplus \alpha_s) \oplus ( \beta_1\oplus \cdots \oplus \beta_t)$, then for the multi-linkoid $\mu$ representing $D$ \[c(\mu)\geq \sum_{i=1}^s\alpha_i + \sum_{i=1}^t|\beta_j|,\] where the sum is taken in $\mathbb{Z}$ by regarding $\alpha_k=0$ or $1$ as an element of $\mathbb{Z}$.

 \label{lem} \end{lemma}
 
 The original statement and proof of Kamada and Oshiro's Lemma \ref{lem} is phrased in terms of triple points in broken sheet diagrams and the triple point number of surface-links  \cite{oshiro}. Since a quandle 3-cocycle of $X$ is a quandle 2-cocycle of $(X,X)$ and the $f$-weight of knotoids is an invariant like the $f$-weight of surface-links, Lemma \ref{lem} is a valid reinterpretation of their result in the context of shadow quandle cocycles and knotoids.

Consider a multi-linkoid diagram realizing the crossing number of the multi-linkoid that it represents. If a 3-cocycle has $\mathbb{Z}$ coefficients and only takes the values $1, -1$ or 0, then the absolute value of the cocycle's weight cannot be greater than the number of crossings in the diagram. Since the weight of a quandle 3-cocycle is an invariant, the crossing number bounds the weight of the cocycle. This is the principle of the lemma's proof.

\begin{theorem}

The shadow quandle cocycle invariant determines the crossing numbers of infinitely many multi-linkoids.
\end{theorem}

\begin{proof}

Take $n$ trivial knotoids diagrams. Generically immerse one copy of $S^1$ that runs under each knotoid diagram once, creating positive crossings, then simply loops back. See Figure \ref{fig:infinite}. Let $M_n$ denote these multi-linkoid diagrams. These diagrams can be $(P_3,P_3)$-colored by coloring all arcs of the $S^1$ component 0, all knotoid diagrams 2, the region adjacent to the heads 2, and the region adjacent to the legs 1. Let $c$ denote this shadow coloring and $\phi$ the 3-cocycle of $P_3$ defined in Example \ref{ex:p3}. There are $n$ positive crossings each colored $(1,0,2)$. Thus, \[ W_\phi(M_n,c)=0\oplus n\in\mathbb{Z}_2\oplus\mathbb{Z}.\]
 Lemma \ref{lem} implies that the crossing number of $M_n$ is $n$.
\end{proof}

\begin{figure}[ht]

\begin{overpic}[unit=.434mm,scale=.8]{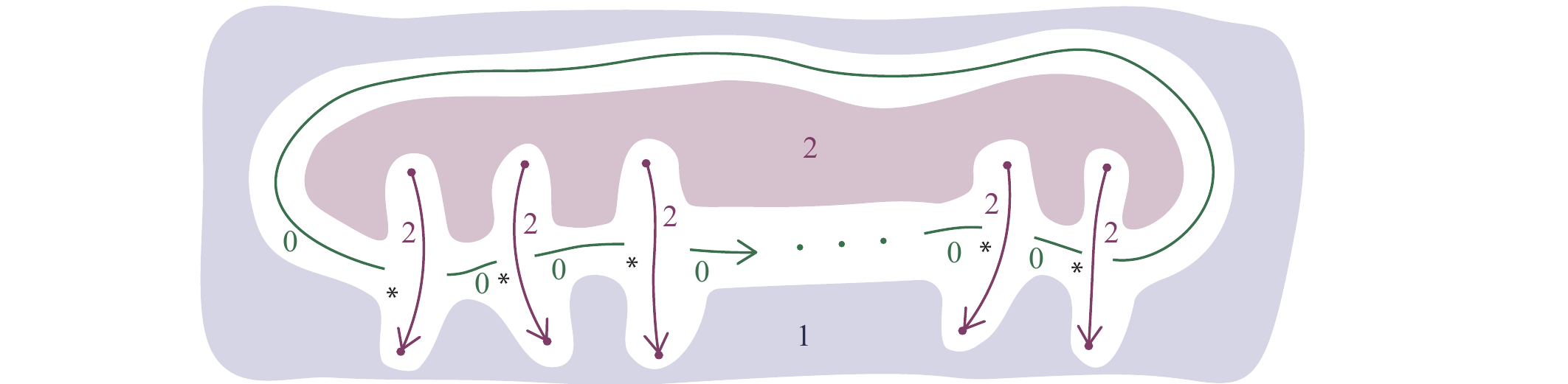}

\end{overpic}

\label{fig:infinite}

\caption{$(P_3,P_3)$-coloring of the multi-linkoids $M_n$.}
\end{figure}

   \bibliographystyle{amsplain}
            \bibliography{railarc}

\end{document}